\documentclass[11pt]{article}
\usepackage{psfrag}
\usepackage{amssymb,bbm}
\usepackage{amsmath}
\usepackage{amsfonts}
\usepackage{graphicx}
\usepackage{yfonts}
\usepackage{setspace}
\usepackage[all]{xy}
\usepackage{pstricks,pst-math,pst-infixplot}
\usepackage[dotinlabels]{titletoc}

\newcounter{theorem}
\newtheorem{theorem}{Theorem}

\newtheorem{lemma}{Lemma}

\newtheorem{example}{Example}

\newtheorem{remark}{Remark}
\newtheorem{definition}{Definition}

\newenvironment{proof}[1][Proof]{\textbf{#1.} }{\rule{0.5em}{0.5em}}

\title{Tight Span of Subsets of The Plane \\ With The Maximum Metric}

\author{Mehmet KILI\c{C}  \footnote{Anadolu
University, Science Faculty, Department of Mathematics, 26470,
Eskişehir, Turkey, \newline e-mails: kompaktuzay@gmail.com,
 skocak@anadolu.edu.tr}  \and \c{S}ahin KO\c{C}AK $^*$ \thanks{Corresponding Author.}}

\begin{document}

\maketitle

\begin{abstract}

 We prove that a nonempty closed and geodesically convex subset
of the $l_{\infty}$ plane $\mathbb{R}^2_{\infty}$ is hyperconvex
and we characterize the tight spans of arbitrary subsets of
$\mathbb{R}^2_{\infty}$ via this property: Given any nonempty
$X\subseteq\mathbb{R}^2_{\infty}$, a closed, geodesically convex
and minimal subset $Y\subseteq\mathbb{R}^2_{\infty}$ containing
$X$
 is isometric to the tight span $T(X)$ of $X$.



\end{abstract}

\section{Introduction}

The notion of hyperconvexity and the associated notion of
hyperconvex or injective hull of metric spaces were introduced in
the two important papers Aronszajn - Panitchpakdi~\cite{aro} and
Isbell~\cite{isb}. About twenty years later Dress~\cite{dre}
rediscovered the injective hull (which he called the tight span)
of metric spaces and opened new ways of looking at the problem of
optimal realizations of finite metric spaces in weighted graphs.
Though that paper of Dress was a turning point, it remained a
notoriously difficult problem to construct the tight span of
finite metric spaces with more than a few points and relate them
to optimal realizations (\cite{Koolen}, \cite{Sturmfels},
\cite{Althofer}). Recently D. Eppstein~\cite{epp} gave an
algorithm which decides whether the tight span of a finite metric
space can be embedded into the $l_1$ plane (i.e. the so-called
Manhattan plane) and he constructed the tight span of a subset of
the Manhattan plane under a certain condition (see Thm~\ref{at-2}
below). In this note we want to characterize (in Thm~\ref{at}) the
tight span of any subset of the $l_{\infty}$-plane (which is
isometric to Manhattan plane) without any restrictions (and
without relying on Eppstein's theorem).

Aronszajn-Panitchpakdi called a metric space $(X,d)$
hyperconvex, if for any collection $(x_i)_{i\in I}$ of points in
$X$ and any collection $(r_i)_{i\in I}$ of nonnegative real
numbers satisfying $d(x_i,x_j)\leq r_i+r_j$ for all $i,j\in I$,
the intersection of closed balls around $x_i$ with radius $r_i$ is
nonempty: $\bigcap_{i\in I}\bar{B}(x_i,r_i)\neq\emptyset$.
($\bar{B}(x_i,r_i)=\{x\in X|\ d(x_i,x)\leq r_i\}$). They showed
that a hyperconvex metric space $X$ is retract of any space $Y$,
where $X$ is isometrically embedded in, whereby the retraction can
be chosen nonexpansive (i.e. distance - non - increasing).

Isbell adopted a more categorical point of view and constructed in
the category, whose objects are metric spaces and whose morphisms
are non-expansive maps, injective objects and injective hulls for
any metric spaces. An injective object in this category is a
metric space $X$, which satisfies the following property: For any
isometric embedding $i:Y\rightarrow Z$ and any morphism
$f:Y\rightarrow X$ in this category, there exists an extension of
$f$ to $Z$; i.e. a morphism $\widetilde{f}:Z\rightarrow X$ such
that the following diagram commutes:

\[
\xymatrix{Y \ar@{^{(}->}[rr]^i \ar[dr]_f && Z \ar@{-->}[dl]^{\widetilde{f}} \\& X &}
\]

Isbell showed that for any metric space $X$ there exists an
injective metric space $\tilde{X}$ with an isometric embedding
$i:X\hookrightarrow \tilde{X}$ such that there is no proper
injective subspace of $\tilde{X}$ containing $i(X)$. He also
showed that this property characterizes $\tilde{X}$ up to
isometry. He called this space the injective envelope of the
metric space $X$. For an excellent survey on hyperconvexity and
injectivity we refer to \cite{esp}. It turns out that a metric
space is hyperconvex if and only if it is injective and
consequently the
 injective hull and the hyperconvex hull of a metric space $X$
(i.e. a minimal hyperconvex space containing $X$) are isometric
objects (see also \cite{kha}). We note for later use the rather
startling property that a metric space is hyperconvex if any
isometric embedding $f:X\rightarrow X\cup \{y\}$, where $y\notin
X$,  has a non-expansive retraction (see~\cite{esp}).

We recall briefly the construction of the injective envelope of
Isbell (or, with another terminology, the ``tight span" of Dress).
Let $X$ be any metric space. Consider the set $T(X)$ of functions
$f:X\rightarrow\mathbb{R}^{\geq0}$ satisfying the following two
properties:

\begin{enumerate}
\item[i)] $f(x)+f(y)\geq d(x,y)$, for all $x,y\in X.$
\item[ii)]$\inf_{y\in X}(f(x)+f(y)-d(x,y))=0$, for all $x\in X.$
\end{enumerate}

The second property implies that the functions satisfying these
properties are minimal in the sense that the point-wise values of
a function $f$ can not be lowered. On the other hand, if $f$ is a
minimal function satisfying the first property (i.e. if $g$ is
another function satisfying the first property and $g\leq f$, then
$g=f$), then $f$ satisfies the second property.

The tight span (or injective envelope) $T(X)$ of $X$ is obtained by putting the supremum metric $d_{\infty}$ on the set $T(X)$:
\[d_{\infty}(f,g)=\sup_{x\in X}|f(x)-g(x)|.\]

\section{Tight Span of Subsets of The $l_{\infty}$ Plane}

D. Epstein gave the following theorem about the tight span of
subsets of the Manhattan plane (Lemma 9 in \cite{epp}):

\begin{theorem} \label{at-2}
Let $X$ be a nonempty subset of the $l_1$ plane (not necessarily
finite). If the orthogonal convex hull $H$ of $X$ is connected,
then $H$ is isometric to the tight span of $X$. (Orthogonal convex
hull is defined to consist of all points surrounded by $X$, and a
point $p$ is the $l_1$ plane is said to be surrounded by $X$ if
each of the four closed axis-aligned quadrants with $p$ as their
apex contains at least one point of $X$.)
\end{theorem}

\begin{remark}
This theorem is not true as it stands, since, for example for an open square,
the orthogonal convex hull equals this open square and is not hyperconvex; hence, it can not be the tight span.
 But D. Eppstein remarks that one can fix the theorem by taking the closure of the orthogonal convex hull (\cite{epp1}).
\end{remark}

It is well-known that the $l_1$ plane is isometric to the $l_{\infty}$ plane (though this is false for higher dimensions) and we prefer, only as a matter of taste, the $l_{\infty}$ plane
 with the maximum metric, since the $l_{\infty}$ spaces are the natural home of tight spans.

We will give below a theorem (\emph{Thm} \ref{at}) characterizing
the tight span of any subset of the $l_{\infty}$ plane. We first
recall that in the $l_{\infty}$ plane, which we denote by
$\mathbb{R}^2_{\infty}$ (i.e. $(\mathbb{R}^2,d_{\infty})$ with
$d_{\infty}((p_1,p_2),(q_1,q_2))=\max\{|p_1-q_1|,|p_2-q_2|\}$),
between any two points $p=(p_1,p_2)$ and $y=(q_1,q_2)$ there exist
paths whose length equal $d_{\infty}(p,q)$. Such a path realizing
the distance between the points $p$ and $q$ is called a geodesic
if it is parameterized by arc length. In this sense
$\mathbb{R}^2_{\infty}$ is a strictly intrinsic metric space in
the terminology of \cite{bur} and geodesic space in the
terminology of \cite{pap}. A subspace
$X\subseteq\mathbb{R}^2_{\infty}$ is called geodesically convex if
for any two points $p,q\in X$, there exists a geodesic in
$\mathbb{R}^2_{\infty}$ which is contained in $X$. In other words,
a subspace $X\subseteq\mathbb{R}^2_{\infty}$ is strictly intrinsic
with respect to the induced metric if and only if it is
geodesically convex.
 We can now formulate the following theorem:

\begin{theorem}
\label{at} Let $X\subseteq\mathbb{R}^2_{\infty}$ be a nonempty
subspace. Let $Y\subseteq\mathbb{R}^2_{\infty}$ be a closed,
geodesically convex subspace containing $X$ and minimal with these
properties. Then $Y$ is isometric to the tight span of $T(X)$ of
$X$.
\end{theorem}

Before giving the proof we note that our assumptions are also
necessary. It is well-known that a hyperconvex metric space is
complete (see \cite{esp}) and we show that it is also strictly
intrinsic (though this property doesn't seem to be noted in the
literature):

\begin{lemma}
\label{at0}
A hyperconvex metric space is strictly intrinsic.
\end{lemma}
(We defer the proof of this lemma to the appendix).

Theorem~\ref{at} is obviously a consequence of the following

\begin{theorem}
\label{at2} A nonempty closed and geodesically convex subspace of
$\mathbb{R}^2_{\infty}$ is hyperconvex.
\end{theorem}

\begin{proof}
Let $A\subseteq\mathbb{R}^2_{\infty}$ be a closed and geodesically
convex subspace. It will be enough to show that $A$ is injective.
Let $A\cup\{z\}$ be an arbitrary one-point extension of the metric
space $A$. We have to show that there exists a nonexpansive
retraction $A\cup\{z\}\rightarrow A$. We first note that it will
be enough to assume $z\in\mathbb{R}^2_{\infty}$. Because,
otherwise we can extend the metric on $A\cup\{z\}$ to
$\mathbb{R}^2_{\infty}\cup\{z\}$ (see Lemma~\ref{at5} in the
appendix) and find by hyperconvexity of $\mathbb{R}^2_{\infty}$ a
nonexpansive retraction
$r:\mathbb{R}^2_{\infty}\cup\{z\}\rightarrow\mathbb{R}^2_{\infty}$.
Consider the point $p=r(z)$ and the embedding $A\hookrightarrow
A\cup\{p\}$. If there is a nonexpansive retraction
$r_p:A\cup\{p\}\rightarrow A$, then $r_p\circ
r|_{A\cup\{z\}}:A\cup\{z\}\rightarrow A$ is a nonexpansive
retraction. So we can work with one-point extensions
$A\hookrightarrow A\cup\{p\}$ for $p\in\mathbb{R}^2_{\infty}$ (in
fact $p\in\mathbb{R}^2_{\infty}\setminus A$).

Before proceeding with the proof, we want to give a few technical
definitions and lemmas (whose proofs we defer to the appendix) we
shall use during the proof.

\begin{definition}
\begin{itemize}
\item[i)]
For $p,q\in \mathbb{R}^n_{\infty}$ we define
\[
D_{pq}=\{u\in \mathbb{R}^n|\
d_\infty(p,u)+d_\infty(u,q)=d_\infty(p,q)\}.
\]
(see Fig.~\ref{fat4} for $n=2$)

$D_{pq}$ is the union of geodesic segments from $p$ to $q$.

\begin{figure}[h]
\begin{center}
\begin{pspicture*}(-1,-1)(4,4)
\psline{->}(-1,0)(3.5,0) \uput[r](3.5,0){$x$}
\psline{->}(0,-1)(0,3.5) \uput[u](0,3.5){$y$}
\pspolygon*[linecolor=lightgray](0.5,1.5)(1.5,0.5)(3,2)(2,3)
\psline(0.5,1.5)(1.5,0.5)(3,2)(2,3)(0.5,1.5) \psdot(0.5,1.5)
\psdot(3,2) \uput[d](0.5,1.5){$p$} \uput[d](3,2){$q$}
\uput[r](2,3){$D_{pq}$}
\end{pspicture*}
\caption{Union of geodesic segments from $p$ to $q$} \label{fat4}
\end{center}
\end{figure}
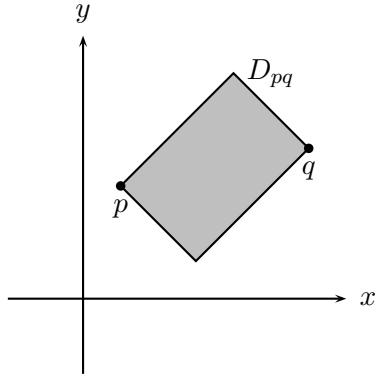

\item[ii)]
For $p\in \mathbb{R}^n_{\infty}$ we define
$S_i^{\varepsilon}(p)=\{q\in\mathbb{R}^n|\
d_{\infty}(p,q)=\varepsilon(q_i-p_i)\}$ for $i=1,2,\cdots,n$ and
$\varepsilon=\pm$ and call them the sectors at the point $p$ (see
Fig.~\ref{fat1})

Note that for $q\in S_1^{\varepsilon}(p)$,
$D_{pq}=S_i^\varepsilon(p)\cap S_i^{-\varepsilon}(q)$.

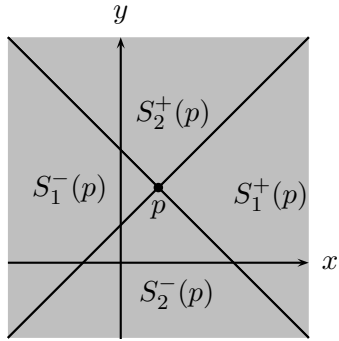
\begin{figure}[h]
\begin{center}
\begin{pspicture*}(-1.5,-1)(3,3.5)
\pspolygon*[linecolor=lightgray](0.5,1)(2.5,3)(2.5,-1)
\psline(0.5,1)(2.5,3) \psline(0.5,1)(2.5,-1)
\pspolygon*[linecolor=lightgray](0.5,1)(2.5,3)(-1.5,3)
\psline(0.5,1)(-1.5,3) \psline(0.5,1)(2.5,3)
\pspolygon*[linecolor=lightgray](0.5,1)(-1.5,-1)(-1.5,3)
\psline(0.5,1)(-1.5,-1) \psline(0.5,1)(-1.5,3)
\pspolygon*[linecolor=lightgray](0.5,1)(-1.5,-1)(2.5,-1)
\psline(0.5,1)(2.5,-1) \psline(0.5,1)(-1.5,-1)
\psline{->}(-1.5,0)(2.5,0) \uput[r](2.5,0){$x$}
\psline{->}(0,-1.5)(0,3) \uput[u](0,3){$y$} \psdot(0.5,1)
\uput[d](0.5,1){$p$} \uput[d](2,1.3){$S_1^+(p)$}
\uput[u](0.7,1.6){$S_2^+(p)$} \uput[l](0,1){$S_1^-(p)$}
\uput[d](0.75,0){$S_2^-(p)$}
\end{pspicture*}
\caption{Sectors of a point $p$ in the $l_{\infty}$ plane}
\label{fat1}
\end{center}
\end{figure}

\item[iii)]
For $p=(p_1,p_2)\in\mathbb{R}^2_{\infty}$ and
$\varepsilon_1,\varepsilon_2=\pm$ we call the set
\[
I^{\varepsilon_1\varepsilon_2}(p)=\{(p_1+\varepsilon_1t,p_2+\varepsilon_2t)|\
t\geq0\}
\]
the $\varepsilon_1\varepsilon_2$-ray at the point $p$ (see
Fig.~\ref{fat5})

\begin{figure}[h]
\begin{center}
\psset{unit=0.75cm}
\begin{pspicture*}(-4,-3)(6,4)
\psline{->}(-4,0)(5,0) \psline{->}(0,-3)(0,3.5) \uput[u](5,0){$x$}
\uput[r](0,3.5){$y$} \psdot(1,0.5) \uput[l](1,0.5){$p$}
\psline(1,0.5)(4,3.5) \uput[r](4,3.5){$I^{++}(p)$}
\psline(1,0.5)(4,-2.5) \uput[r](4,-2.5){$I^{+-}(p)$}
\psline(1,0.5)(-2,-2.5) \uput[l](-2,-2.5){$I^{--}(p)$}
\psline(1,0.5)(-2,3.5) \uput[l](-2,3.5){$I^{-+}(p)$}
\end{pspicture*}
\caption{The $\varepsilon_1\varepsilon_2$-rays at the point $p$}
\label{fat5}
\end{center}
\end{figure}
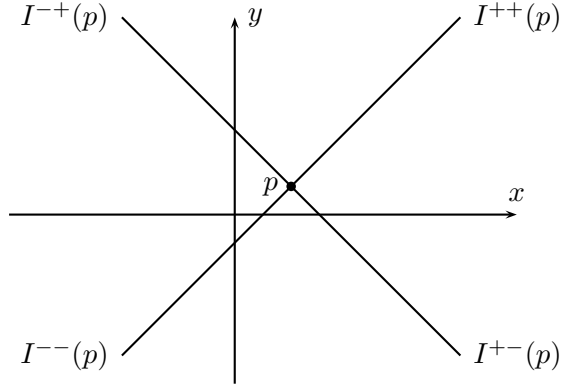
\end{itemize}
\end{definition}

Note that
$I^{\varepsilon_1\varepsilon_2}(p)=S_1^{\varepsilon_1}(p)\cap
S_2^{\varepsilon_2}(p)$.

\begin{lemma}
\label{at2.5}
Let $u\in\mathbb{R}^2_{\infty}$, $p\in S_1^{\varepsilon}(u)$,
$q\in S_2^{\delta}(u)$ and $\gamma$ a geodesic between the points
$p$ and $q$. Then there exists a $t$ (in the domain of definition
of $\gamma$) such that $\gamma(t)\in I^{\varepsilon\delta}(u)$.
\end{lemma}

\begin{lemma}
\label{at3} Let $A\subseteq \mathbb{R}^2_{\infty}$ a geodesically
convex subspace and $p\in\mathbb{R}^2_{\infty}$. If each sector
$S_i^{\varepsilon}(p)$ of the point $p$ contains a point of $A$,
then $p$ belongs to the set $A$.
\end{lemma}

\begin{lemma}
\label{at4} Let $A\subseteq\mathbb{R}^n_{\infty}$ a connected
subspace and $p\in\mathbb{R}^n$. If two opposite sectors
$S_i^+(p)$ and $S_i^-(p)$ of the point $p$ intersect the set $A$,
but no other sectors of $p$ intersect $A$, then $p\in A$.
\end{lemma}

Now we continue with the proof of Theorem~\ref{at2}.

Let any point $p=(p_1,p_2)\in\mathbb{R}^2\setminus A$ be given. We
have to construct a nonexpansive retraction $A\cup\{p\}\rightarrow
A$. We consider three cases, depending on how many sectors of $p$
intersect the set $A$.

Three-Sectors Case:

Let us assume that three sectors of $p$ intersect $A$. Without
loss of generality we can take the sectors $S_1^+(p)$, $S_2^+(p)$
and $S_2^-(p)$. Since $A$ is closed and connected (as a
geodesically convex subspace), the ray $\{(p_1+t,p_2)|\ t\geq0\}$
intersects the set $A$ at a point with minimal $t$, say $t_0$
(i.e. the first intersection point). Denote this point by $q=(q_1,
q_2)$. In the interior of the sector $S_1^-(q)$ there can be no
point of the set $A$.
 To see this, assume to the contrary that there exists a point $u\in A$ lying in this region. Without loss of generality we can assume that $u$ lies above the line
  $pq$ (i.e. with an ordinate higher than that of p). Now consider a point $v\in A\cap S_2^{-}(p)$. By geodesical convexity of $A$ there is a geodesic between $u$ and $v$,
   which must intersect the line $pq$. But this produces a point $w\in A$ left to the point $q$, which contradicts the choice of $q$ (see Fig.~\ref{fat6}).

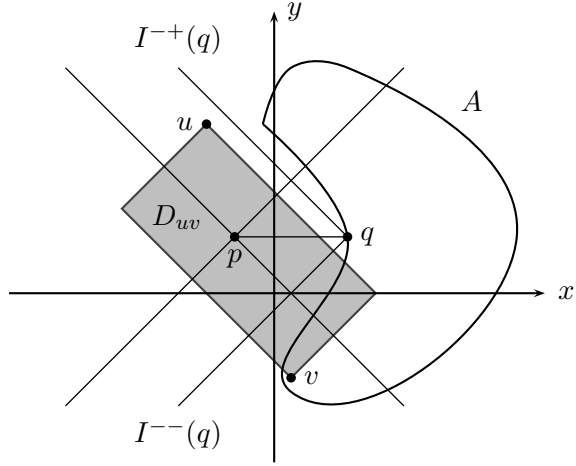
\begin{figure}[h!]
\begin{center}
\psset{unit=0.75cm}
\begin{pspicture*}(-5,-3)(5,5.5)
\pspolygon*[linecolor=lightgray](0,-1.5)(-3,1.5)(-1.5,3)(1.5,0)(0,-1.5)
\psline[linecolor=darkgray](0,-1.5)(-3,1.5)(-1.5,3)(1.5,0)(0,-1.5)
\uput[r](-2.7,1.35){$D_{uv}$}
\pscurve(-0.5,3)(1,1)(0,-1.75)(4,1)(1,4)(0,4)(-0.5,3)
\psline{->}(-5,0)(4.5,0)
\psline{->}(-0.3,-3)(-0.3,5)
\psline[linewidth=0.5pt](-4,-2)(2,4)
\psline[linewidth=0.5pt](-4,4)(2,-2)
\uput[r](4.5,0){$x$}
\uput[r](-0.3,5){$y$}
\uput[d](-1,1){$p$}
\psdot(-1,1)
\psdot(1,1)
\uput[r](1,1){$q$}
\psdot(-1.5,3)
\uput[l](-1.5,3){$u$}
\psdot(0,-1.5)
\uput[r](0,-1.5){$v$}
\uput[u](3.2,3){$A$}
\psline[linewidth=0.5pt](1,1)(-1,1)
\psline[linewidth=0.5pt](-2,4)(1,1)
\psline[linewidth=0.5pt](-2,-2)(1,1)
\uput[d](-2,-2){$I^{--}(q)$}
\uput[u](-2,4){$I^{-+}(q)$}
\end{pspicture*}
\caption{Three-sectors case} \label{fat6}
\end{center}
\end{figure}

We define the function $r:A\cup\{p\}\rightarrow A$,

 \begin{equation*}
 r(x)=\left\{\begin{array}{cll}
 x &,& x\in A\\
 q &,& x=p .\\
 \end{array}\right.
 \end{equation*}
The function $r$ is a nonexpansive retraction. To see this let $a=(a_1,a_2)\in A$. We saw above that $a\notin (S_1^-(q))^{\circ}$.
There are now three possibilities: $a\in S_1^+(q)$, $a\in S_2^+(q)$ and $a\in S_2^-(q)$.

If $a\in S_1^+(q)$; then
\begin{eqnarray*}
d_\infty(a,q)&=&a_1-q_1=a_1-(p_1+t_0)\\
&\leq&a_1-p_1= d_\infty(a,p).
\end{eqnarray*}

If $a\in S_2^+(q)$; then

\begin{eqnarray*}
d_\infty(a,q)=a_2-q_2=a_2-p_2\leq d_\infty(a,p).
\end{eqnarray*}

If $a\in S_2^-(q)$; then

\begin{eqnarray*}
d_\infty(a,q)=q_2-a_2=p_2-a_2\leq d_\infty(a,p).
\end{eqnarray*}

So, in all cases $d_{\infty}(a,q)\leq d_{\infty}(a,p)$, as claimed.

Two-Sectors Case:

Let us assume that two sectors of $p$ intersect $A$. Since $A$ is connected, these two sectors can not be opposite sectors by Lemma~\ref{at4}.
So, without loss of generality we can assume that the sectors are $S_1^+(p)$ and $S_2^+(p)$. First note that the line $\{(p_1+t,p_2-t)|\ t\in {\mathbb{R}}\}$ can not intersect the set $A$,
 since otherwise $A$ would intersect at least three sectors of $p$. Now imagine that we move this line along the ray $I^{++}(p)$ towards $A$ (i.e. consider the lines $I^{+-}(p_1+t,p_2+t)\cup I^{-+}(p_1+t,p_2+t)$ for $t\geq 0$).
 Let $t_0$ be the supremum of the parameters $t$, for which the corresponding lines do not intersect $A$. Let $q=(q_1,q_2)=(p_1+t_0,p_2+t_0)$ (see Fig.~\ref{fat7}).

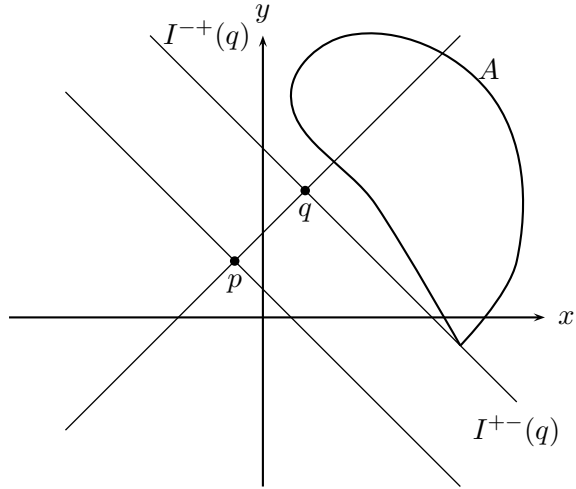
\begin{figure}[h]
\begin{center}
\begin{pspicture*}(-5,-3)(5,5.5)
\psset{unit=0.75cm}
\pscurve(3,-0.5)(4,1)(3.5,4)(1,5)(0,4)(1.5,2)(3,-0.5)
\psline{->}(-5,0)(4.5,0)
\psline{->}(-0.5,-3)(-0.5,5)
\psline[linewidth=0.5pt](-4,-2)(-1,1)
\psline[linewidth=0.5pt](-1,1)(3,5)
\psline[linewidth=0.5pt](-4,4)(3,-3)
\uput[r](4.5,0){$x$}
\uput[u](-0.5,5){$y$}
\psdot(-1,1)
\psdot(0.25,2.25)
\uput[d](-1,1){$p$}
\uput[d](0.25,2.25){$q$}
\uput[u](3.5,4){$A$}
\psline[linewidth=0.5pt](-2.5,5)(4,-1.5)
\uput[d](4,-1.5){$I^{+-}(q)$}
\uput[r](-2.5,5){$I^{-+}(q)$}
\end{pspicture*}
\caption{Two-sectors case} \label{fat7}
\end{center}
\end{figure}

   Now we define $r:A\cup\{p\}\rightarrow A\cup\{q\}$,

 \begin{equation*}
 r(x)=\left\{\begin{array}{cll}
 x &,& x\in A\\
 q &,& x=p\\
 \end{array}\right.
 \end{equation*}
 The function $r$ is nonexpansive. To see this, let $a=(a_1,a_2)\in A$. The point $a$ can belong to $S_1^+(q)$ or $S_2^+(q)$.

 If $a\in S_1^+(q)$; then

 \begin{eqnarray*}
 d_\infty(a,q)=a_1-q_1=a_1-(p_1+t_0)\leq a_1-p_1=d_\infty(a,p).
 \end{eqnarray*}

If $a\in S_2^+(q)$; then

 \begin{eqnarray*}
 d_\infty(a,q)=a_2-q_2=a_2-(p_2+t_0)\leq a_2-p_2=d_\infty(a,p).
 \end{eqnarray*}

 So we have $d_\infty(a,q)\leq d_\infty(a,p)$, as claimed.

 Now if $q\in A$, then $r:A\cup\{p\}\rightarrow A$ is a nonexpansive retraction and we are done.

 If $q\notin A$, then there are two possibilities. Either the line $I^{-+}(q)\cup I^{+-}(q)$ intersects $A$ or it does not intersect $A$. If it intersects $A$,
 it can not intersect both of the rays $I^{+-}(q)$ and $I^{-+}(q)$. Because otherwise the unique geodesic between two such points would contain $q$ and thus $q$
 would belong to $A$. So assume without loss of generality that $A$ intersects $I^{+-}(q)$  (see Fig.~\ref{fat7}). Now consider the ray $\{(q_1+t,q_2)|\ t>0\}$ and denote
 its first intersection point with $A$ by $q'$. The set $A$ intersects three sectors of $q$ and by the proof of the first case we have a nonexpansive retraction
 $r':A\cup\{q\}\rightarrow A$ mapping $q\mapsto q'$. Combining the two nonexpansive retractions we get a retraction $r'\circ r:A\cup\{p\}\rightarrow A$.
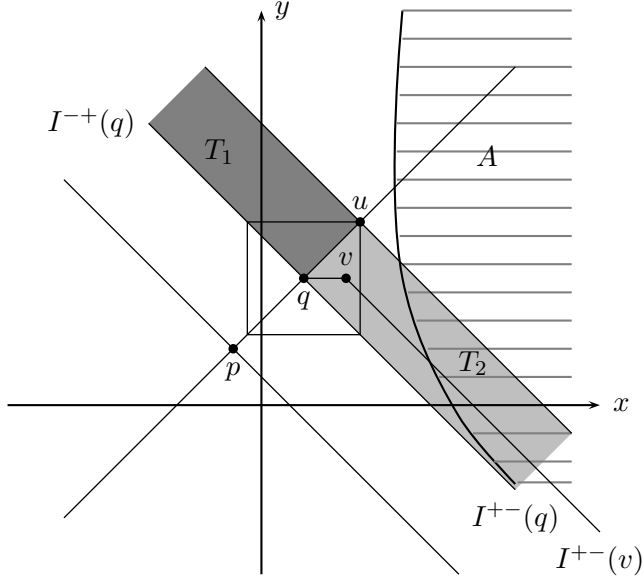
\begin{figure}[h]
 \begin{center}
\begin{pspicture*}(-4,-3)(6,7)
\psset{unit=0.75cm}
\pspolygon*[linecolor=gray](-1.5,6)(1.25,3.25)(0.25,2.25)(-2.5,5)
\pspolygon*[linecolor=lightgray](1.25,3.25)(5,-0.5)(4,-1.5)(0.25,2.25)
\pscurve(2,7)(2,2.25)(3,-0.2)(4,-1.4)
\psline[linecolor=gray](2,7)(5,7)
\psline[linecolor=gray](2,6.5)(5,6.5)
\psline[linecolor=gray](1.95,6)(5,6)
\psline[linecolor=gray](1.95,5.5)(5,5.5)
\psline[linecolor=gray](1.9,5)(5,5)
\psline[linecolor=gray](1.9,4.5)(5,4.5)
\psline[linecolor=gray](1.9,4)(5,4)
\psline[linecolor=gray](1.9,3.5)(5,3.5)
\psline[linecolor=gray](1.95,3)(5,3)
\psline[linecolor=gray](2,2.5)(5,2.5)
\psline[linecolor=gray](2.1,2)(5,2)
\psline[linecolor=gray](2.25,1.5)(5,1.5)
\psline[linecolor=gray](2.45,1)(5,1)
\psline[linecolor=gray](2.65,0.5)(5,0.5)
\psline[linecolor=gray](3.25,-0.5)(5,-0.5)
\psline[linecolor=gray](3.62,-1)(5,-1)
\psline[linecolor=gray](4.04,-1.37)(5,-1.37)
\psline{->}(-5,0)(5.5,0) \psline{->}(-0.5,-3)(-0.5,7)
\psline[linewidth=0.5pt](-4,-2)(-1,1)
\psline[linewidth=0.5pt](-1,1)(4,6)
\psline[linewidth=0.5pt](-4,4)(3,-3) \uput[r](5.5,0){$x$}
\uput[r](-0.5,7){$y$} \psdot(-1,1) \psdot(0.25,2.25)
\uput[d](-1,1){$p$} \uput[d](0.25,2.25){$q$}
\psline[linewidth=0.5pt](-0.75,1.25)(1.25,1.25)(1.25,3.25)(-0.75,3.25)(-0.75,1.25)
\psdot(1.25,3.25) \uput[u](1.25,3.25){$u$} \psdot(1,2.25)
\uput[u](1,2.25){$v$} \psline[linewidth=0.5pt](0.25,2.25)(1,2.25)
\psline[linewidth=0.5pt](1,2.25)(5.5,-2.25)
\uput[d](5.5,-2.25){$I^{+-}(v)$} \uput[u](3.25,0.3){$T_2$}
\uput[l](-0.75,4.5){$T_1$}
\psline[linewidth=0.5pt](-1.5,6)(5,-0.5) \uput[u](3.5,4){$A$}
\psline[linewidth=0.5pt](-2.5,5)(4,-1.5)
\uput[d](4,-1.5){$I^{+-}(q)$} \uput[l](-2.5,5){$I^{-+}(q)$}
\end{pspicture*}
\caption{Two-sectors sub-case} \label{fat8}
\end{center}
\end{figure}

 We now consider the case where the line $I^{-+}(q)\cup I^{+-}(q)$ does not intersect $A$ (see Fig.~\ref{fat8}).
 Since $A$ is closed there is a closed ball $\bar{B}(q,\varepsilon_0)$ not intersecting $A$. Let $u=(q_1+\varepsilon_0,q_2+\varepsilon_0)$ and consider the stripe
 $T_1$ bounded by the two rays $I^{-+}(q)$, $I^{-+}(u)$ and
 the segment $[qu]$ and the stripe $T_2$ bounded by the rays $I^{+-}(q)$, $I^{+-}(u)$ and the segment $[qu]$. The set $A$ intersects one and only one of these stripes.
  It intersects one of them by the definition $q$ and it can not intersect both of them, because otherwise a geodesic between two such points would intersect
  the segment $[qu]$, contradicting the choice of $u$, and we can assume without loss of generality that $A$ intersects the stripe $T_2$. Now take any point $v=(q_1+\varepsilon,q_2)$
  inside $\bar{B}(q,\varepsilon_0)$ such that the ray $I^{+-}(v)$ intersects $A$.
 We are now in the position of the three-sectors case with respect to the point $v$. Now combining the nonexpansive functions $r:A\cup\{p\}\rightarrow A\cup\{q\}$,
  $r_1:A\cup\{q\}\rightarrow A\cup\{v\}$ and $r_2:A\cup\{v\}\rightarrow A$, we get a nonexpansive retraction $r_2\circ r_1\circ r:A\cup\{p\}\rightarrow A$.

 One-Sector Case:

 Now we consider the final case, where only one sector of $p$ intersects $A$ and we can assume this sector to be $S_1^+(p)$ (In fact, A must then lie in the interior of this sector).
 Imagine that we move the ''right elbow'' $I^{++}(p)\cup I^{+-}(p)$ of $p$ horizontally to the right, i.e. we consider the elbows $I^{++}(p_1+t,p_2)\cup I^{+-}(p_1+t,p_2)$ for $t\geq0$. Let $t_0$ denote the
 supremum of the parameters $t$, for which the corresponding elbows do not intersect $A$. Denote $q=(p_1+t_0,p_2)$. If $q\in A$, then we are done, since we get a nonexpansive retraction $r:A\cup\{p\}\rightarrow A$,
  sending $p$ to $q$. If $q\notin A$, but the elbow $I^{++}(q)\cup I^{+-}(q)$ of $q$ intersects $A$, then we are in a position of two-sectors case or three-sectors case with respect to $q$ and
   combining the nonexpansive functions $r:A\cup\{p\}\rightarrow A\cup\{q\}$ (sending $p$ to $q$) and $r':A\cup\{q\}\rightarrow A$, we get a nonexpansive retraction $r'\circ r:A\cup\{p\}\rightarrow A$ and we are done.

\begin{figure}[h]
\begin{center}
\begin{pspicture*}(-5,-2)(6,7.5)
\psset{unit=0.85cm}
\pscurve(4.1,6)(3.25,5)(2.5,4)(2,2.5)(3,2.75)(4,2.5)(5.5,3)
\psline[linecolor=lightgray](4.14,6)(5.5,6)
\psline[linecolor=lightgray](3.7,5.5)(5.5,5.5)
\psline[linecolor=lightgray](3.3,5)(5.5,5)
\psline[linecolor=lightgray](2.85,4.5)(5.5,4.5)
\psline[linecolor=lightgray](2.5,4)(5.5,4)
\psline[linecolor=lightgray](2.2,3.5)(5.5,3.5)
\psline[linecolor=lightgray](2,3)(5.5,3)
\pspolygon*[linecolor=lightgray](1,1.5)(1,3)(-0.5,3)(-0.5,1.5)(1,1.5)
\psline{->}(-5,0)(5.5,0)
\psline{->}(0,-2)(0,7) \psline[linewidth=0.5pt](0.25,2.25)(4,6)
\uput[r](5.5,0){$x$} \uput[u](0,7){$y$} \psdot(0.25,2.25)
\uput[d](0.25,2.25){$q$}
\psline[linewidth=0.5pt](-1.75,2.25)(0.75,2.25) \psdot(0.75,2.25)
\uput[d](0.75,2.25){$v$}
\psline[linewidth=0.5pt](0.75,2.25)(4.5,6)
\psline[linewidth=0.5pt](0.75,2.25)(4.5,-1.5)
\psline[linewidth=0.5pt](1,1.5)(1,3)(-0.5,3)(-0.5,1.5)(1,1.5)
\uput[u](4.5,4){$A$} \psline[linewidth=0.5pt](0.25,2.25)(4,-1.5)
\psdot(-1.75,2.25) \psline[linewidth=0.5pt](-5,-1)(2.5,6.5)
\psline[linewidth=0.5pt](-5,5.5)(2.5,-2) \uput[d](-1.75,2.25){$p$}
\uput[r](-5,-1){$I^{--}(p)$} \uput[r](-4.5,5){$I^{-+}(p)$}
\uput[r](2,6){$I^{++}(p)$} \uput[r](2,-1.5){$I^{+-}(p)$}
\end{pspicture*}
\caption{One-sector case} \label{fat10}
\end{center}
\end{figure}
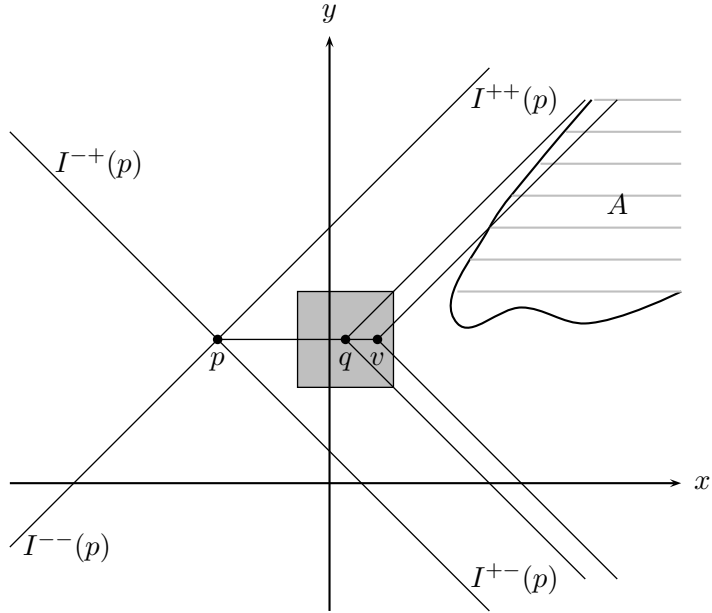

     Now we consider the case where the elbow $I^{++}(q)\cup I^{+-}(q)$ of $q$ does not intersect $A$
     (see Fig.~\ref{fat10}). As $q\notin A$ and $A$ is closed, there is a ball $\bar{B}(q,\varepsilon)$ not intersecting $A$.
     Choose a $\delta$ with $0<\delta<\varepsilon$ such that the
     right elbow $I^{++}(v)\cup I^{+-}(v)$ of the point
     $v=(q_1+\delta,q_2)$ intersects $A$. The map $r':A\cup\{q\}\rightarrow
     A\cup\{v\}$, sending $q$ to $v$ and fixing the points of
     $A$ is obviously a nonexpansive function.

     Since the set $A$ intersects two or three sectors of $v$, we can construct a nonexpansive retraction $r'':A\cup\{v\}\rightarrow A$ and combining $r$, $r'$ and $r''$, we get a nonexpansive
     retraction $r''\circ r'\circ r:A\cup\{p\}\rightarrow A$.
\end{proof}

\section{Some Examples}

Using Theorem~\ref{at} one can produce easily many examples of
tight spans of (finite or infinite) subspaces of the plane
$\mathbb{R}_{\infty}^2$. We give below several examples.

\begin{example}
Tight span of a three-point metric space $X=\{P_1,P_2,P_3\}$ with
$d(P_2,P_3)=a$, $d(P_1,P_3)=b$ and $d(P_1,P_2)=c$. Assume without
loss of generality $a\geq b$. Note that this space can be embedded
into $\mathbb{R}_{\infty}^2$, e.g. as in Fig.~\ref{fat11} (the
point $p_i$ being the image of $P_i$ under this embedding).

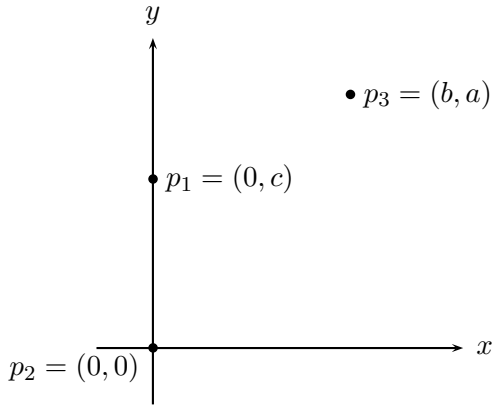
\begin{figure}[h!]
\begin{center}
\begin{pspicture}(-1,-1)(6,6)
\psset{unit=0.75}
\psline{->}(-1,0)(5.5,0) \psline{->}(0,-1)(0,5.5)
\uput[r](5.5,0){$x$} \uput[u](0,5.5){$y$}
\psdots(0,0)(0,3)(3.5,4.5)
\uput[l](0,-0.35){$p_2=(0,0)$}
\uput[r](0,3){$p_1=(0,c)$}
\uput[r](3.5,4.5){$p_3=(b,a)$}
\end{pspicture}
\caption{Embedding of a three-point metric space with side-lengths
$a, b, c$ in $\mathbb{R}^2_{\infty}$} \label{fat11}
\end{center}
\end{figure}

Now consider the set $T\subseteq \mathbb{R}_{\infty}^2$ in
Fig.~\ref{fat12} (where the bold segments are parallel to the
diagonals).

\begin{figure}[h!]
\vskip-20pt
\begin{center}
\begin{pspicture}(-1,-1)(6,6)
\psset{unit=0.9}
\psline{->}(-1,0)(5.5,0) \psline{->}(0,-1)(0,5.5)
\uput[r](5.5,0){$x$} \uput[u](0,5.5){$y$}
\psdots(0,0)(0,3)(3.5,4.5)(1,2)
\psline[linewidth=1.5pt](0,0)(1.5,1.5)(0,3)
\psline[linewidth=1.5pt](3.5,4.5)(1,2) \uput[l](0,-0.35){$p_2=(0,0)$}
\uput[l](0,3){$p_1=(0,c)$} \uput[r](3.5,4.5){$p_3=(b,a)$}
\uput[r](1.1,2.1){$q=(\frac{b+c-a}{2},\frac{a+c-b}{2})$}
\uput[u](2,4){$T$}
\end{pspicture}
\vskip-10pt \caption{Tight span of the three-point metric space in
Fig.\ref{fat11} inside $\mathbb{R}^2_{\infty}$} \label{fat12}
\end{center}
\end{figure}
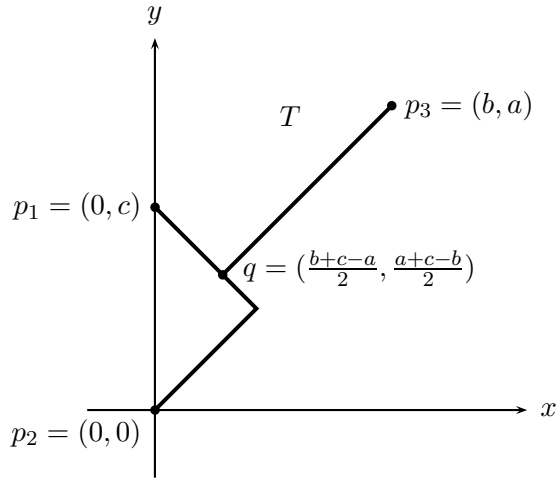

T is a closed, geodesically convex set containing $X$ and minimal
with these properties. So, it is isometric to the tight span of
$X$.

Other realizations of the tight span of $X$ could be, for example
the set $S$ in Fig.~\ref{fat13}, or $U$ in Fig.~\ref{fat14},
since $S$ and $U$ are also closed, geodesically convex and minimal
sets containing $X$. Obviously $T$, $S$ and $U$ are isometric
spaces (as they should be, since they are tight spans of $X$).

\begin{figure}[h!]
\begin{center}
\begin{pspicture}(-1,-1)(6,6)
\psset{unit=0.75}
\psline{->}(-1,0)(5.5,0) \psline{->}(0,-1)(0,5.5)
\uput[r](5.5,0){$x$} \uput[u](0,5.5){$y$}
\psdots(0,0)(0,3)(3.5,4.5)(1,2)
\psline[linewidth=1.5pt](0,0)(1,2)(0,3)
\psline[linewidth=1.5pt](3.5,4.5)(1,2) \uput[l](0,-0.35){$p_2=(0,0)$}
\uput[l](0,3){$p_1=(0,c)$} \uput[r](3.5,4.5){$p_3=(b,a)$}
\uput[r](1,2){$q=(\frac{b+c-a}{2},\frac{a+c-b}{2})$}
\uput[u](2,4){$S$}
\end{pspicture}
\caption{Another (isometric) tight span of the same three-point
metric space inside $\mathbb{R}^2_{\infty}$} \label{fat13}
\end{center}
\end{figure}
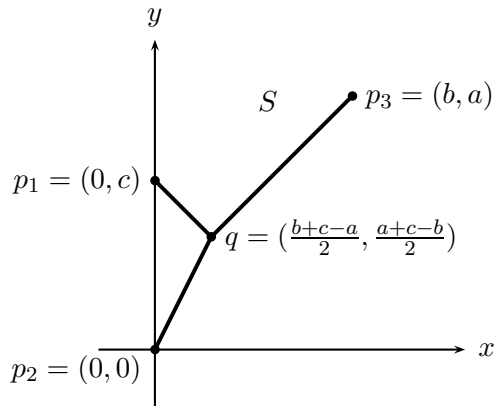

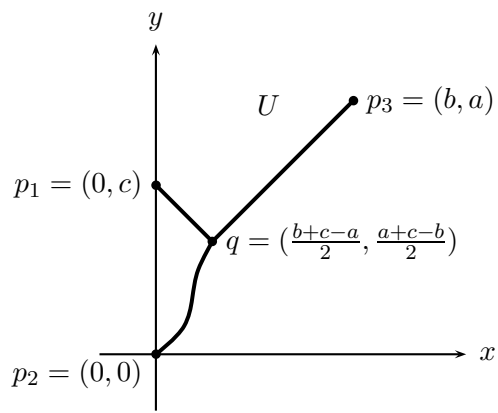
\begin{figure}[h!]
\begin{center}
\begin{pspicture}(-1,-1)(6,6)
\psset{unit=0.75}
\psline{->}(-1,0)(5.5,0) \psline{->}(0,-1)(0,5.5)
\uput[r](5.5,0){$x$} \uput[u](0,5.5){$y$}
\psdots(0,0)(0,3)(3.5,4.5)(1,2) \psline[linewidth=1.5pt](1,2)(0,3)
\psline[linewidth=1.5pt](3.5,4.5)(1,2)
\pscurve[linewidth=1.5pt](0,0)(0.5,0.5)(0.75,1.5)(1,2)
\uput[l](0,-0.35){$p_2=(0,0)$} \uput[l](0,3){$p_1=(0,c)$}
\uput[r](3.5,4.5){$p_3=(b,a)$}
\uput[r](1,2){$q=(\frac{b+c-a}{2},\frac{a+c-b}{2})$}
\uput[u](2,4){$U$}
\end{pspicture}
\caption{Still another (isometric) tight span of the same
three-point metric space inside $\mathbb{R}^2_{\infty}$}
\label{fat14}
\end{center}
\end{figure}

\end{example}

\clearpage

\begin{example}
Tight span of a four point metric space $X=\{P_1,P_2,P_3,P_4\}$ with distances as shown in Fig.~\ref{fat15}.
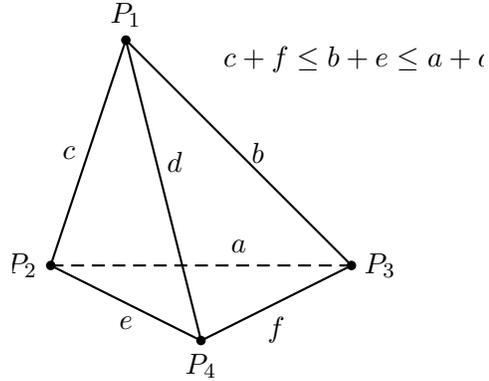
\begin{figure}[h]
\begin{center}
\psset{unit=0.5}
\begin{pspicture*}(-2,-2)(10.5,8)
\psdots(1,7)(-1,1)(3,-1)(7,1) \uput[u](1,7){$P_1$}
\uput[l](-1,1){$P_2$} \uput[r](7,1){$P_3$} \uput[d](3,-1){$P_4$}
\psline(1,7)(-1,1)(3,-1)(1,7) \psline(1,7)(7,1)(3,-1)
\psline[linestyle=dashed](-1,1)(7,1) \uput[l](0,4){$c$}
\uput[r](1.75,3.75){$d$} \uput[r](4,4){$b$} \uput[u](4,1){$a$}
\uput[d](1,0){$e$} \uput[d](5,0){$f$} \uput[r](3.25,6.5){$c+f\leq
b+e\leq a+d$}
\end{pspicture*}
\caption{A four-point metric space} \label{fat15}
\end{center}
\end{figure}

 We can always arrange (by renaming the points) that the inequalities \linebreak \mbox{$c+f\leq b+e\leq a+d$} hold. This space can be embedded into $\mathbb{R}_{\infty}^2$, e.g. as in Fig.~\ref{fat15.5}.

\begin{figure}[h!]
\vskip-10pt
\begin{center}
\begin{pspicture}(-3,-1)(8,8)
\psset{unit=0.75}
\psline{->}(-2,0)(7.5,0) \psline{->}(0,-1)(0,7.5)
\uput[r](7.5,0){$x$} \uput[u](0,7.5){$y$}
\psdots(0,0)(-1.5,4.5)(4.5,6.5)(5.75,3.25)
\uput[l](0,-0.35){$p_2=(0,0)$}
\uput[l](-1.5,4.5){$p_1=(e-d,c)$}
\uput[r](4.5,6.5){$p_3=(b+e-d,a)$}
\uput[r](5.75,3.25){$p_4=(e,a-f)$}
\end{pspicture}
\caption{Embedding of the four-point metric space in
Fig.\ref{fat15} in  $\mathbb{R}^2_{\infty}$} \label{fat15.5}
\end{center}
\end{figure}
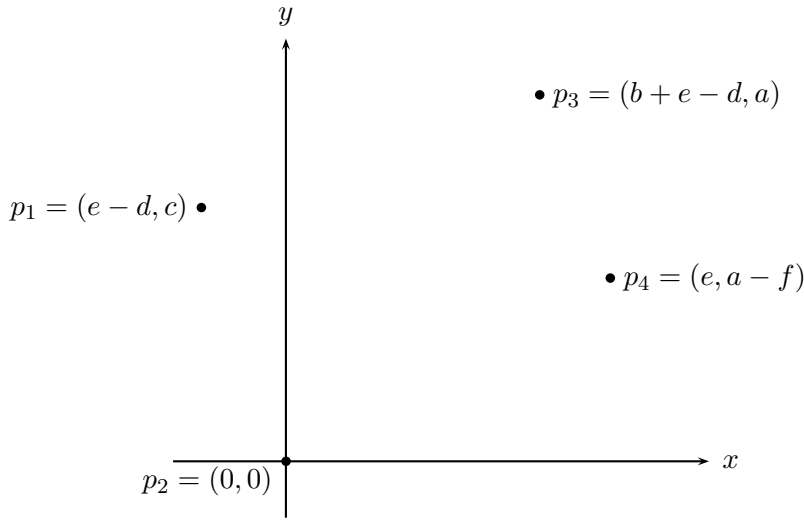

Now consider the set $T\subseteq \mathbb{R}_{\infty}^2$ in
Fig.~\ref{fat16} (where the bold segments are parallel to the
diagonals).

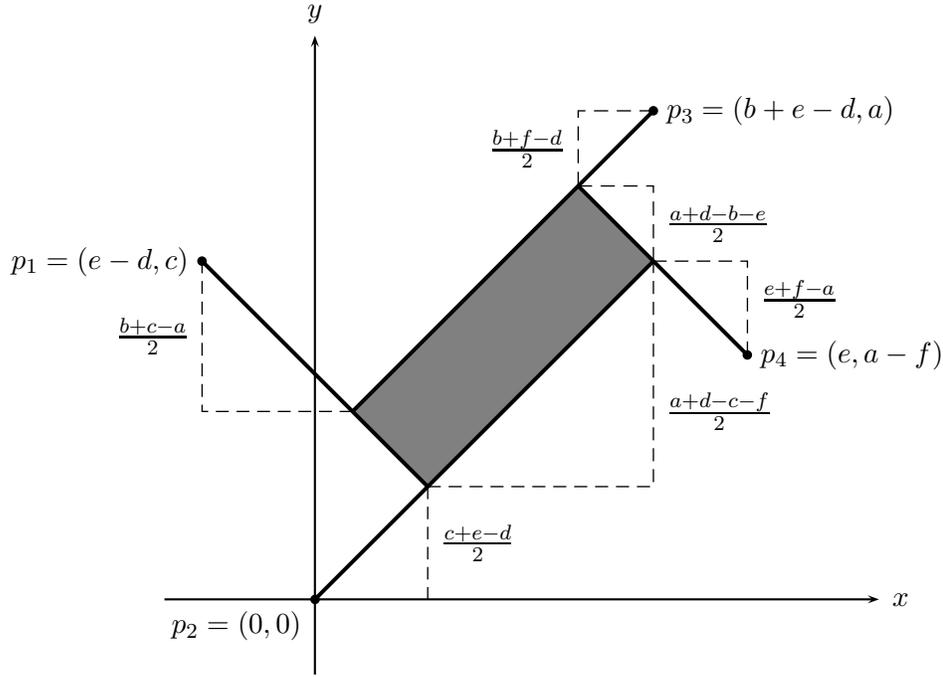
\begin{figure}[h!]
\begin{center}
\begin{pspicture}(-3,-1)(8,8)
\pspolygon*[linecolor=gray](1.5,1.5)(4.5,4.5)(3.5,5.5)(0.5,2.5)
\psline{->}(-2,0)(7.5,0) \psline{->}(0,-1)(0,7.5)
\uput[r](7.5,0){$x$} \uput[u](0,7.5){$y$}
\psdots(0,0)(-1.5,4.5)(4.5,6.5)(5.75,3.25)
\psline[linewidth=1.5pt](0,0)(1.5,1.5)(-1.5,4.5)
\psline[linewidth=1.5pt](0.5,2.5)(4.5,6.5)
\psline[linewidth=1.5pt](3.5,5.5)(5.75,3.25)
\psline[linewidth=1.5pt](4.5,4.5)(1.5,1.5)
\uput[l](0,-0.35){$p_2=(0,0)$}
\uput[l](-1.5,4.5){$p_1=(e-d,c)$}
\uput[r](4.5,6.5){$p_3=(b+e-d,a)$}
\uput[r](5.75,3.25){$p_4=(e,a-f)$}
\psline[linestyle=dashed, linewidth=0.5pt](1.5,1.5)(1.5,0)
\uput[r](1.5,0.75){$\frac{c+e-d}{2}$}
\psline[linestyle=dashed, linewidth=0.5pt](1.5,1.5)(4.5,1.5)(4.5,4.5)
\uput[r](4.5,2.5){$\frac{a+d-c-f}{2}$}
\psline[linestyle=dashed, linewidth=0.5pt](5.75,3.25)(5.75,4.5)(4.5,4.5)
\uput[r](5.75,4){$\frac{e+f-a}{2}$}
\psline[linestyle=dashed, linewidth=0.5pt](4.5,4.5)(4.5,5.5)(3.5,5.5)
\uput[r](4.5,5){$\frac{a+d-b-e}{2}$}
\psline[linestyle=dashed, linewidth=0.5pt](3.5,5.5)(3.5,6.5)(4.5,6.5)
\uput[l](3.5,6){$\frac{b+f-d}{2}$}
\psline[linestyle=dashed, linewidth=0.5pt](0.5,2.5)(-1.5,2.5)(-1.5,4.5)
\uput[l](-1.5,3.5){$\frac{b+c-a}{2}$}
\end{pspicture}
\caption{Tight span of the four-point metric space in
Fig.\ref{fat15.5} inside  $\mathbb{R}^2_{\infty}$} \label{fat16}
\end{center}
\end{figure}
T is a closed, geodesically convex set containing $X$ and minimal with these properties. So, it is isometric to the tight span of $X$.
\end{example}

\begin{example}
In this sample of examples we show the tight spans of some
infinite subsets of $\mathbb{R}_{\infty}^2$ (see the Figs.
\ref{fat17}-\ref{fat22}). In each case, by Theorem~\ref{at}, it is
enough to see that the corresponding spaces $T_i$ are closed,
geodesically convex subspaces containing the given spaces $X_i$
and minimal with these properties. During checking these
properties one should bear in mind that for two points on a line
parallel to a diagonal, there exists a unique geodesic between
these points and it is the segment between these points.

\begin{figure}[h!]
\begin{center}
\begin{pspicture}(-0.5,-5)(11.5,3)
\psline{->}(-0.5,0)(2.5,0) \psline{->}(0,-0.5)(0,2.5)
\uput[u](2.5,0){$x$} \uput[u](0,2.5){$y$}
\psarc[linewidth=1.5pt](0,0){2}{0}{90} \uput[u](1.65,1.3){$X_1$}
\psline{->}(3.5,0)(7,0) \psline{->}(4.5,-0.5)(4.5,2.5)
\uput[u](7,0){$x$} \uput[u](4.5,2.5){$y$}
\psarc[linewidth=1.5pt](4.5,0){2}{20}{110} \uput[u](6,1.4){$X_2$}
\psline{->}(8,0)(11.5,0) \psline{->}(9.5,-0.5)(9.5,2.5)
\uput[u](11.5,0){$x$} \uput[u](9.5,2.5){$y$}
\psarc[linewidth=1.5pt](9.5,0){2}{45}{135}
\uput[u](10.75,1.6){$X_3$}

\psline{->}(-0.5,-4.5)(2.5,-4.5) \psline{->}(0,-5)(0,-2)
\uput[u](2.5,-4.5){$x$} \uput[u](0,-2){$y$}
\psarc*[linecolor=gray](0,-4.5){2}{0}{90} \psarc(0,-4.5){2}{0}{90}
\psline(0,-2.5)(2,-4.5) \uput[u](1.65,-3.2){$T_1$}
\psline{->}(3.5,-4.5)(7,-4.5) \psline{->}(4.5,-5)(4.5,-2)
\uput[u](7,-4.5){$x$} \uput[u](4.5,-2){$y$}
\psarc*[linecolor=gray](4.5,-4.5){2}{20}{70}
\psarc(4.5,-4.5){2}{20}{70}
\psarc[linewidth=1.5pt](4.5,-4.5){2}{70}{110}
\psline(6.379,-3.816)(5.184,-2.62) \uput[u](6,-3.1){$T_2$}
\psline{->}(8,-4.5)(11.5,-4.5) \psline{->}(9.5,-5)(9.5,-2)
\uput[u](11.5,-4.5){$x$} \uput[u](9.5,-2){$y$}
\psarc[linewidth=1.5pt](9.5,-4.5){2}{45}{135}
\uput[u](10.75,-2.9){$T_3$}
\end{pspicture}
\vskip-10pt \caption{Three subsets $X_1, X_2$ and $X_3\subset
\mathbb{R}^2_{\infty}$ and their tight spans $T_1, T_2$ and $T_3$}
\label{fat17}
\end{center}
\end{figure}

\begin{figure}[h!]
\begin{center}
\begin{pspicture}(-1,-6)(11.5,3)
\psline{->}(-1,0.5)(2.5,0.5) \psline{->}(0.75,-1)(0.75,2.5)
\uput[u](2.5,0.5){$x$} \uput[u](0.75,2.5){$y$}
\psarc[linewidth=1.5pt](0.75,0.5){1.5}{0}{180}
\uput[u](1.65,1.8){$X_1$} \psline{->}(3.25,0.5)(7,0.5)
\psline{->}(5,-1)(5,2.5) \uput[u](7,0.5){$x$} \uput[u](5,2.5){$y$}
\psarc[linewidth=1.5pt](5,0.5){1.5}{20}{200}
\uput[u](6.25,1.6){$X_2$} \psline{->}(7.75,0.5)(11.5,0.5)
\psline{->}(9.5,-1)(9.5,2.5) \uput[u](11.5,0.5){$x$}
\uput[u](9.5,2.5){$y$}
\psarc[linewidth=1.5pt](9.5,0.5){1.5}{45}{225}
\uput[u](10.5,1.8){$X_3$}

\psline{->}(-1,-4.5)(2.5,-4.5) \psline{->}(0.75,-6)(0.75,-2.5)
\uput[u](2.5,-4.5){$x$}
\uput[u](0.75,-2.5){$y$}
\psarc*[linecolor=gray](0.75,-4.5){1.5}{0}{90}
\psarc*[linecolor=gray](0.75,-4.5){1.5}{90}{180}
\psarc(0.75,-4.5){1.5}{0}{180}
\psline(2.25,-4.5)(0.75,-3)(-0.75,-4.5)
\uput[u](1.65,-3.2){$T_1$}
\uput[u](7,-4.5){$x$}
\uput[u](5,-2.5){$y$}
\psarc*[linecolor=gray](5,-4.5){1.5}{20}{70}
\psarc*[linecolor=gray](5,-4.5){1.5}{70}{200}
\psarc(5,-4.5){1.5}{20}{200}
\psline(6.41,-3.986)(5.513,-3.09)(3.59,-5.013)
\uput[u](6.25,-3.4){$T_2$}
\psline{->}(3.25,-4.5)(7,-4.5) \psline{->}(5,-6)(5,-2.5)
\uput[u](11.5,-4.5){$x$}
\uput[u](9.5,-2.5){$y$}
\psarc*[linecolor=gray](9.5,-4.5){1.5}{45}{225}
\psarc(9.5,-4.5){1.5}{45}{225}
\psline(10.56,-3.439)(8.439,-5.56)
\uput[u](10.5,-3.2){$T_3$}
\psline{->}(7.75,-4.5)(11.5,-4.5) \psline{->}(9.5,-6)(9.5,-2.5)
\end{pspicture}
\vskip-10pt \caption{Three subsets $X_1, X_2$ and $X_3\subset
\mathbb{R}^2_{\infty}$ and their tight spans $T_1, T_2$ and $T_3$}
\label{fat18}
\end{center}
\end{figure}

\begin{figure}[h!]
\begin{center}
\begin{pspicture}(-1,-6)(11.5,3)
\psset{unit=0.9} \psline{->}(-1,0.5)(2.5,0.5)
\psline{->}(0.75,-1.25)(0.75,2.5) \uput[u](2.5,0.5){$x$}
\uput[u](0.75,2.5){$y$}
\psarc[linewidth=1.5pt](0.75,0.5){1.5}{0}{270}
\uput[u](1.65,1.8){$X_1$} \psline{->}(3.25,0.5)(7,0.5)
\psline{->}(5,-1.25)(5,2.5) \uput[u](7,0.5){$x$}
\uput[u](5,2.5){$y$} \psarc[linewidth=1.5pt](5,0.5){1.5}{20}{290}
\uput[u](6.1,1.65){$X_2$} \psline{->}(7.75,0.5)(11.5,0.5)
\psline{->}(9.5,-1.25)(9.5,2.5) \uput[u](11.5,0.5){$x$}
\uput[u](9.5,2.5){$y$}
\psarc[linewidth=1.5pt](9.5,0.5){1.5}{45}{315}
\uput[u](10.75,1.6){$X_3$}

\uput[u](2.5,-4.25){$x$}
\uput[l](0.75,-2.25){$y$}
\psarc*[linecolor=gray](0.75,-4.25){1.5}{0}{270}
\psarc(0.75,-4.25){1.5}{0}{270}
\psline(2.25,-4.25)(0.75,-5.75)
\uput[u](1.65,-3){$T_1$}
\psline{->}(-1,-4.25)(2.5,-4.25) \psline{->}(0.75,-6)(0.75,-2.25)
\uput[u](7,-4.25){$x$}
\uput[l](5,-2.25){$y$}
\psarc*[linecolor=gray](5,-4.25){1.5}{20}{250}
\psarc*[linecolor=gray](5,-4.25){1.5}{160}{290}
\psarc(5,-4.25){1.5}{20}{290}
\psline(6.41,-3.737)(5,-5.147)(5.513,-5.66)
\uput[u](6.1,-3.1){$T_2$}
\psline{->}(3.25,-4.25)(7,-4.25) \psline{->}(5,-6)(5,-2.25)
\uput[u](11.5,-4.25){$x$}
\uput[l](9.5,-2.25){$y$}
\psarc*[linecolor=gray](9.5,-4.25){1.5}{45}{225}
\psarc*[linecolor=gray](9.5,-4.25){1.5}{135}{315}
\psarc(9.5,-4.25){1.5}{45}{315}
\psline(10.56,-3.189)(9.5,-4.25)(10.56,-5.31)
\uput[u](10.65,-3.1){$T_3$}
\psline{->}(7.75,-4.25)(11.5,-4.25) \psline{->}(9.5,-6)(9.5,-2.25)
\end{pspicture}
\vskip-25pt \caption{Three subsets $X_1, X_2$ and $X_3\subset
\mathbb{R}^2_{\infty}$ and their tight spans $T_1, T_2$ and $T_3$}
\label{fat19}
\end{center}
\end{figure}

\begin{figure}[h!]
\begin{center}
\begin{pspicture}(-1,-6)(11.5,3)
\psset{unit=0.9}
\psline{->}(-1,0.5)(2.5,0.5) \psline{->}(0.75,-1.25)(0.75,2.5)
\uput[u](2.5,0.5){$x$}
\uput[u](0.75,2.5){$y$}
\psline[linewidth=1.5pt](1.85,1.6)(-0.35,-0.6)
\psline[linewidth=1.5pt](1.85,-0.6)(-0.35,1.6)
\uput[u](1.65,1.65){$X_1$}
\psline{->}(3.25,0.5)(7,0.5) \psline{->}(5,-1.25)(5,2.5)
\uput[u](7,0.5){$x$}
\uput[u](5,2.5){$y$}
\psline[linewidth=1.5pt](5.75,2)(4.25,-1)
\psline[linewidth=1.5pt](3.5,1.25)(6.5,-0.25)
\uput[u](6.25,1.6){$X_2$}
\psline{->}(7.75,0.5)(11.5,0.5) \psline{->}(9.5,-1.25)(9.5,2.5)
\uput[u](11.5,0.5){$x$}
\uput[u](9.5,2.5){$y$}
\psline[linewidth=1.5pt](9.5,2)(9.5,-1)
\psline[linewidth=1.5pt](11,0.5)(8,0.5)
\uput[u](10,1.3){$X_3$}

\psline{->}(-1,-4.25)(2.5,-4.25) \psline{->}(0.75,-6)(0.75,-2.25)
\uput[u](2.5,-4.25){$x$}
\uput[l](0.75,-2.25){$y$}
\psline[linewidth=1.5pt](1.85,-3.15)(-0.35,-5.35)
\psline[linewidth=1.5pt](1.85,-5.35)(-0.35,-3.15)
\uput[u](1.65,-3){$T_1$}

\psline(5.75,-2.75)(4.5,-4)(3.5,-3.5)(4.75,-4.75)(4.25,-5.75)(5.5,-4.5)(6.5,-5)(5.25,-3.75)(5.75,-2.75)
\pspolygon*[linecolor=gray](5.75,-2.75)(4.5,-4)(3.5,-3.5)(4.75,-4.75)(4.25,-5.75)(5.5,-4.5)(6.5,-5)(5.25,-3.75)(5.75,-2.75)
\psline(5.75,-2.75)(4.5,-4)(3.5,-3.5)(4.75,-4.75)(4.25,-5.75)(5.5,-4.5)(6.5,-5)(5.25,-3.75)(5.75,-2.75)
\psline{->}(3.25,-4.25)(7,-4.25) \psline{->}(5,-6)(5,-2.25)
\uput[u](7,-4.25){$x$}
\uput[l](5,-2.25){$y$}
\uput[u](6.25,-3.15){$T_2$}

\uput[u](11.5,-4.25){$x$}
\uput[l](9.5,-2.25){$y$}
\pspolygon*[linecolor=gray](11,-4.25)(9.5,-2.75)(8,-4.25)(9.5,-5.75)
\psline(11,-4.25)(9.5,-2.75)(8,-4.25)(9.5,-5.75)(11,-4.25)
\psline{->}(7.75,-4.25)(11.5,-4.25) \psline{->}(9.5,-6)(9.5,-2.25)
\uput[u](10.25,-3.25){$T_3$}
\end{pspicture}
\vskip-20pt \caption{Three subsets $X_1, X_2$ and $X_3\subset
\mathbb{R}^2_{\infty}$ and their tight spans $T_1, T_2$ and $T_3$}
\label{fat20}
\end{center}
\end{figure}

\begin{figure}[h!]
\begin{center}
\psset{unit=0.9}
\begin{pspicture}(-1,-6)(11.5,3)
\psset{unit=0.9}
\psline{->}(-1,0.5)(2.5,0.5) \psline{->}(0.75,-1.25)(0.75,2.5)
\uput[u](2.5,0.5){$x$}
\uput[u](0.75,2.5){$y$}
\psline[linewidth=1.5pt](2.25,0.5)(-0.75,0.5)
\psline[linewidth=1.5pt](1.5,1.8)(0,-0.8)
\psline[linewidth=1.5pt](0,1.8)(1.5,-0.8)
\uput[u](1.65,1.8){$X_1$}
\psline{->}(3.25,0.5)(7,0.5) \psline{->}(5,-1.25)(5,2.5)
\uput[u](7,0.5){$x$}
\uput[u](5,2.5){$y$}
\psline[linewidth=1.5pt](6.449,0.888)(3.55,0.11)
\psline[linewidth=1.5pt](5.388,1.949)(4.61,-0.95)
\psline[linewidth=1.5pt](3.94,1.56)(6.06,-0.56)
\uput[u](6.1,1.65){$X_2$}
\psline{->}(7.75,0.5)(11.5,0.5) \psline{->}(9.5,-1.25)(9.5,2.5)
\uput[u](11.5,0.5){$x$}
\uput[u](9.5,2.5){$y$}
\psline[linewidth=1.5pt](10.8,1.25)(8.2,-0.25)
\psline[linewidth=1.5pt](9.5,2)(9.5,-1)
\psline[linewidth=1.5pt](8.2,1.25)(10.8,-0.25)
\uput[u](10.75,1.6){$X_3$}

\uput[u](2.5,-4.25){$x$}
\uput[l](0.75,-2.25){$y$}
\pspolygon*[linecolor=gray](2.25,-4.25)(1.3,-3.3)(1.5,-2.95)(0.75,-3.616)(0,-2.95)(0.2,-3.3)(-0.75,-4.25)(0.2,-5.2)(0,-5.55)(0.75,-4.884)(1.5,-5.55)(1.3,-5.2)(2.25,-4.25)
\pspolygon(2.25,-4.25)(1.3,-3.3)(1.5,-2.95)(0.75,-3.616)(0,-2.95)(0.2,-3.3)(-0.75,-4.25)(0.2,-5.2)(0,-5.55)(0.75,-4.884)(1.5,-5.55)(1.3,-5.2)(2.25,-4.25)
\uput[u](1.65,-3){$T_1$}
\psline{->}(-1,-4.25)(2.5,-4.25)
\psline{->}(0.75,-6)(0.75,-2.25)

\psline[linewidth=1.5pt](3.94,-3.19)(6.06,-5.31)
\pspolygon*[linecolor=gray](6.449,-3.862)(5.388,-2.8)(3.55,-4.64)(4.61,-5.7)
\pspolygon(6.449,-3.862)(5.388,-2.8)(3.55,-4.64)(4.61,-5.7)
\psline{->}(3.25,-4.25)(7,-4.25) \psline{->}(5,-6)(5,-2.25)
\uput[u](7,-4.25){$x$}
\uput[l](5,-2.25){$y$}
\uput[u](6.1,-3.1){$T_2$}

\pspolygon*[linecolor=gray](10.8,-3.5)(10.45,-3.7)(9.5,-2.75)(8.55,-3.7)(8.2,-3.5)(8.95,-4.25)(8.2,-5)(8.55,-4.8)(9.5,-5.75)(10.45,-4.8)(10.8,-5)(10.05,-4.25)
\psline{->}(7.75,-4.25)(11.5,-4.25) \psline{->}(9.5,-6)(9.5,-2.25)
\uput[u](11.5,-4.25){$x$}
\uput[l](9.5,-2.25){$y$}
\pspolygon(10.8,-3.5)(10.45,-3.7)(9.5,-2.75)(8.55,-3.7)(8.2,-3.5)(8.95,-4.25)(8.2,-5)(8.55,-4.8)(9.5,-5.75)(10.45,-4.8)(10.8,-5)(10.05,-4.25)
\uput[u](10.75,-3.1){$T_3$}
\end{pspicture}
\vskip-10pt \caption{Three subsets $X_1, X_2$ and $X_3\subset
\mathbb{R}^2_{\infty}$ and their tight spans $T_1, T_2$ and $T_3$}
\label{fat21}
\end{center}
\end{figure}

\begin{figure}[h!]
\begin{center}
\begin{pspicture}(-1,-6)(11.5,3)
\psset{unit=0.9} \psline{->}(-1,0.5)(2.5,0.5)
\psline{->}(0.75,-1.25)(0.75,2.5) \uput[u](2.5,0.5){$x$}
\uput[u](0.75,2.5){$y$}
\psellipse[linewidth=1.5pt](0.75,0.5)(1,1.5)
\uput[u](1.5,1.8){$X_1$} \psline{->}(3.25,0.5)(7,0.5)
\psline{->}(5,-1.25)(5,2.5) \uput[u](7,0.5){$x$}
\uput[u](5,2.5){$y$}
\parabola[linewidth=1.5pt](3.5,2.25)(5,-1)
\uput[r](6.4,1.75){$X_2$} \psline{->}(7.75,0.5)(11.5,0.5)
\psline{->}(9.5,-1.25)(9.5,2.5) \uput[u](11.5,0.5){$x$}
\uput[u](9.5,2.5){$y$}
\pscurve[linewidth=1.5pt](9.4,1.2)(9.45,1.6)(9.15,2)(8.5,1.8)(8.25,0.5)(9,-0.65)(10,-1)
\pscurve[linewidth=1.5pt](9.4,1.2)(10,1.8)(10.6,1.9)(10.85,1.5)(9.55,-0.75)(10,-1)
\uput[u](11,1.75){$X_3$}

\uput[u](2.5,-4.25){$x$} \uput[l](0.75,-2.25){$y$}
\psellipse*[linecolor=gray](0.75,-4.25)(1,1.5)
\psellipse(0.75,-4.25)(1,1.5) \uput[u](1.5,-3){$T_1$}
\psline{->}(-1,-4.25)(2.5,-4.25) \psline{->}(0.75,-6)(0.75,-2.25)
\uput[u](7,-4.25){$x$} \uput[l](5,-2.25){$y$}
\psparabola*[linecolor=gray](3.5,-2.5)(5,-5.75)
\psparabola(3.5,-2.5)(5,-5.75) \uput[r](6.3,-3.25){$T_2$}
\psline{->}(3.25,-4.25)(7,-4.25) \psline{->}(5,-6)(5,-2.25)
\uput[u](11.5,-4.25){$x$} \uput[l](9.5,-2.25){$y$}
\pscurve*[linecolor=gray](9.4,-3.55)(9.45,-3.15)(9.15,-2.75)(8.5,-2.95)(8.25,-4.25)(9,-5.35)(9.9,-5.74)(10,-5.75)(9.98,-5.73)(9.55,-5.5)(10.85,-3.25)(10.6,-2.85)(10,-2.95)(9.4,-3.55)
\pscurve(9.4,-3.55)(9.45,-3.15)(9.15,-2.75)(8.5,-2.95)(8.25,-4.25)(9,-5.35)(9.9,-5.74)(10,-5.75)(9.98,-5.73)(9.55,-5.5)(10.85,-3.25)(10.6,-2.85)(10,-2.95)(9.4,-3.55)
\pspolygon*[linecolor=gray](10,-5.75)(9,-4.75)(8.6,-4.75)(9.5,-5.6)(10,-5.75)
\pspolygon*[linecolor=gray](9.24,-2.8)(10.24,-3.8)(9,-3.8)(9.24,-2.8)
\psline(10,-5.75)(9.5,-5.25) \psline(9.24,-2.8)(9.69,-3.25)
\uput[u](11,-3){$T_3$} \psline{->}(7.75,-4.25)(11.5,-4.25)
\psline{->}(9.5,-6)(9.5,-2.25)
\end{pspicture}
\caption{Three subsets $X_1, X_2$ and $X_3\subset
\mathbb{R}^2_{\infty}$ and their tight spans $T_1, T_2$ and $T_3$}
\label{fat22}
\end{center}
\end{figure}
\end{example}

\clearpage

\section{What about $\mathbb{R}_{\infty}^n$?}

It is tempting to hope that the Theorem~\ref{at} would hold
$\mathbb{R}_{\infty}^n$ for $n\geq3$ also. This is however
unfortunately not true as we show in Example~\ref{atex} below. The
main reason is that, Theorem~\ref{at2}, on which the
Theorem~\ref{at} is based, is not true either.

\begin{example}
The plane $L=\{(x,y,z)|\ x+y+z=0\}\subseteq\mathbb{R}_{\infty}^3$ with the induced metric is not hyperconvex. So a nonempty, closed and geodesically convex subspace of $\mathbb{R}_{\infty}^3$ need not be hyperconvex.
\end{example}

To see this, note that the discs around a point in $L$ are regular
hexagons (see Fig.~\ref{23}) and they don't satisfy the
hyperconvexity condition (see Fig.~\ref{24}).
\begin{figure}[h!]
\begin{center}
\begin{pspicture*}(-2,-0.5)(8,6.5)
\psset{unit=0.75}
\pspolygon*[linecolor=lightgray](2,0)(5,1)(6,4)(4,6)(1,5)(0,2)(2,0)
\psline(2,0)(5,1)(6,4)(4,6)(1,5)(0,2)(2,0)
\psdots(2,0)(5,1)(6,4)(4,6)(1,5)(0,2)(2,0)
\psline(0,4)(4,4)(4,0)
\psline(0,4)(2,6)(6,6)(6,2)(4,0)
\psline(6,6)(4,4)
\psline(4,0)(0,0)(0,4)
\psline[linewidth=0.5pt, linestyle=dashed](0,0)(2,2)(6,2)
\psline[linewidth=0.5pt, linestyle=dashed](2,2)(2,6)
\psdot(3,3)
\uput[d](3,3){$O$}
\psline{->}(3,3)(2.65,2.65)
\uput[d](2.55,2.75){$x$}
\psline{->}(3,3)(3.5,3)
\uput[r](3.45,3){$y$}
\psline{->}(3,3)(3,3.5)
\uput[u](3,3.45){$z$}
\uput[d](2,0){$(1,0,-1)$}
\uput[r](5,1){$(0,1,-1)$}
\uput[r](6,4){$(-1,1,0)$}
\uput[u](4,6){$(-1,0,1)$}
\uput[l](1,5){$(0,-1,1)$}
\uput[l](0,2){$(1,-1,0)$}
\end{pspicture*}
\caption{hexagonal slice of a cube}
 \label{23}
\end{center}
\end{figure}
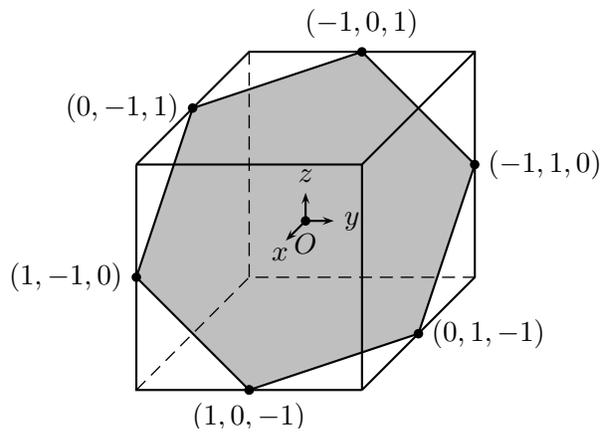


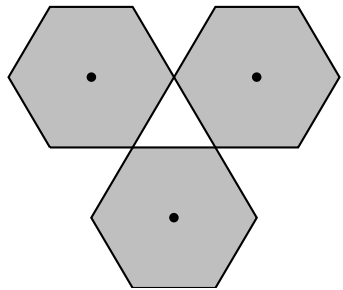
\begin{figure}[h!]
\begin{center}
\psset{unit=0.55}
\begin{pspicture*}(-2,0)(8,7)
\pspolygon*[linecolor=lightgray](2,0)(4,0)(5,1.7)(4,3.4)(2,3.4)(1,1.7)(2,0)
\psline(2,0)(4,0)(5,1.7)(4,3.4)(2,3.4)(1,1.7)(2,0)
\pspolygon*[linecolor=lightgray](4,3.4)(6,3.4)(7,5.1)(6,6.8)(4,6.8)(3,5.1)(4,3.4)
\psline(4,3.4)(6,3.4)(7,5.1)(6,6.8)(4,6.8)(3,5.1)(4,3.4)
\pspolygon*[linecolor=lightgray](0,3.4)(2,3.4)(3,5.1)(2,6.8)(0,6.8)(-1,5.1)(0,3.4)
\psline(0,3.4)(2,3.4)(3,5.1)(2,6.8)(0,6.8)(-1,5.1)(0,3.4)
\psdots(3,1.7)(5,5.1)(1,5.1)
\end{pspicture*}
\caption{Pairwise intersecting hexagons with empty
intersection}
 \label{24}
\end{center}
\end{figure}

\begin{example}
\label{atex} Let $X=\{A=(1,1,1),\ B=(1,-2,-2),\
C=(-1,0,1)\}\subseteq \mathbb{R}_{\infty}^3$. Then the set $Y$
shown in Fig.~\ref{25} is a closed, geodesically convex subspace
containing $X$ and minimal with these properties. But $Y$ is not
isometric to the tight span $T(X)$ of $X$. (T(X) is the union of
the segments $[AO]\cup [BO]\cup [CO]\subset
\mathbb{R}_{\infty}^3$.)
\end{example}
\begin{figure}[h!]
\begin{center}
\psset{unit=0.65}
\begin{pspicture}(-1,-1)(14,9.5)
\pspolygon*[linecolor=lightgray](3,3)(9,9)(9,3)(12,0)(3,3)
\psdots(0,0)(3,3)(9,9)(9,3)(12,0)
\psline(0,0)(3,3)(9,9)(9,3)(12,0)(3,3)
\uput[r](0,0){$B=(1,-2,-2)$} \uput[r](9,9){$A=(1,1,1)$}
\uput[d](12,0){$C=(-1,0,1)$} \uput[r](3.25,3.15){$(1,-1,-1)$}
\uput[u](3,5){$Y$}
\uput[r](9,3){$(-\frac{1}{3},-\frac{1}{3},\frac{1}{3})$}
\uput[l](6,6){$2$} \uput[l](1.5,1.5){$1$} \uput[d](7.5,1.5){$2$}
\uput[r](10.6,1.5){$\frac{2}{3}$} \uput[r](9,6){$\frac{4}{3}$}
\end{pspicture}
\caption{A counterexample in $\mathbb{R}^3_{\infty}$}
 \label{25}
\end{center}
\end{figure}
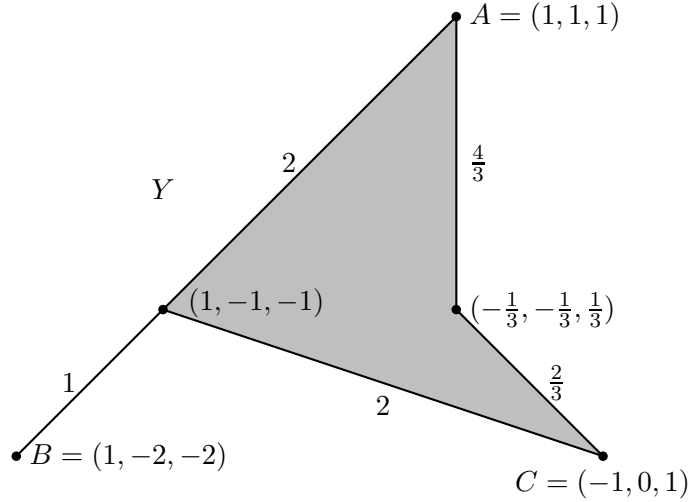

\section{Appendix}

\begin{proof}(of Lemma~\ref{at0})

A complete metric space $(X,d)$ is strictly intrinsic if for every
$x,y\in X$ there exists a midpoint i.e. a point $z\in X$ such that
$d(x,z)=d(z,y)=\frac{1}{2}d(x,y)$ (see~\cite{bur}, Theorem
2.4.16). Since a hyperconvex metric space is complete, so it will
be enough to show that midpoints exist. Assume to the contrary
that for some $x,y\in X$ no midpoint exists. Now consider an
external point $z\notin X$ and define on $X\cup\{z\}$ a metric
satisfying $d(x,z)=d(y,z)=\frac{1}{2}d(x,y)$ (to achieve this the
Lemma~\ref{at5} below can be used taking $Y=\{x,y,z\}$). The
hyperconvex metric space $(X,d)$ is injective and hence there
exists a nonexpansive retraction $r:X\cup\{z\}\rightarrow X$.
 Now, consider the point $r(z)=z'\in X$. By nonexpansiveness we get
\[
d(x,z')\leq d(x,z)=\frac{1}{2}d(x,y)
\]
and
\[
d(z',y)\leq d(z,y)=\frac{1}{2}d(x,y),
\]
which show that $z'$ is a midpoint of $x$ and $y$ contradicting our assumption.
\end{proof}

\begin{proof}(of Lemma~\ref{at2.5})

If one of the points $p$ or $q$ belongs to the set
$S_1^{\varepsilon}(u)\cap
S_2^{\delta}(u)=I^{\varepsilon\delta}(u)$, then we are done. So,
let us assume $p\in S_1^{\varepsilon}(u)\setminus
I^{\varepsilon\delta}(u)$ and $q\in S_2^{\delta}(u)\setminus
I^{\varepsilon\delta}(u)$. In that case $D_{pq}\setminus
I^{\varepsilon\delta}(u)$ becomes a disconnected set, since
$D_{pq}\cap {S_1^{\varepsilon}(u)}^\circ$ and $D_{pq}\cap
{S_2^{\delta}(u)}^\circ$ give a disjoint decomposition of this
set. The points $p$ and $q$ belong to different components of
$D_{pq}\setminus I^{\varepsilon\delta}(u)$ and hence there must be
some $t$ with $\gamma(t)\in I^{\varepsilon\delta}(u)$.
\end{proof}

\begin{proof}(of Lemma~\ref{at3})

Let $a_1\in S_1^+(p)$ and $a_2\in S_2^+(p)$ be points on $A$.
Since $A$ is geodesically convex there exist a geodesic $\gamma$
inside $A$ connecting $a_1$ and $a_2$. According to Lemma
~\ref{at2.5} there exist a $t_0$ with $\gamma(t_0)\in
I^{++}(p)\cap A$. If  $a_3\in S_1^-(p)$ and $a_4\in S_2^-(p)$ are
points on $A$ then there exist a geodesic $\alpha$ inside $A$
connecting these points and a $t_1$ with $\alpha(t_1)\in
I^{--}(p)\cap A$. So, we can write $\gamma(t_0)=(p_1+t,p_2+t)$ and
$\alpha(t_1)=(p_1-k,p_2-k)$ for some $t,k\geq0$. As the only
geodesic connecting the points $\gamma(t_0)$ and $\alpha(t_1)$ is
the segment $[\gamma(t_0), \alpha(t_1)]$, we get $p\in A$.
\end{proof}

\begin{proof}(of Lemma~\ref{at4})

Assume $p\notin A$. Then $S_i^+(p)\cap A$ and $S_i^-(p)\cap A$ constitute a disjoint  and open decomposition of $A$ contradicting the connectivity of $A$.
\end{proof}

\begin{lemma}
\label{at5} Let $(X,d_X)$ and $(Y,d_Y)$ be metric spaces with
$X\cap Y\neq\emptyset$, and assume $d_X|_{X\cap Y}=d_Y|_{X\cap
Y}$. Then there exists  a metric $d$ on $X\cup Y$, such that
$d|_X=d_X$ and $d|_Y=d_Y$.
\end{lemma}

\begin{proof}
One can take, for example,
\[
d(x,y):=\inf_{u\in X\cap Y}\{d(x,u)+d(u,y)\}
\]
\end{proof}

\end{document}